\documentclass[a4paper, 12pt]{amsart}
\usepackage{amsmath}
\usepackage{amssymb}
\usepackage{verbatim}

\theoremstyle{plain}
\newtheorem{assumption}{Assumption}
\newtheorem{main}{Theorem}
\newtheorem{lemma}{Lemma}
\newtheorem{proposition}{Proposition}

\numberwithin{equation}{section}

\theoremstyle{definition}
\newtheorem{definition}{Definition}

\theoremstyle{remark}
\newtheorem*{remark}{Remark}
\newtheorem*{acknowledgment}{Acknowledgment}
\title{ Convergence of global solutions for some classes of nonlinear damped wave equations}
\author{Zhe Jiao \\ \emph{School of Mathematical Science, Fudan University} \\ \emph{Handan Road 220, 200433 Shanghai, P. R. China} \\ \emph{Email: zhejiao@yahoo.com}}
\date{August 6, 2013}

\begin{document}
\maketitle

\pagestyle{plain}

\def\abstractname{\textbf{Abstract}}
\begin{abstract}
  We consider the asymptotic behavior of the solution to the wave equation with time-dependent damping and analytic nonlinearity. Our main goal is to prove the convergence of a global solution to an equilibrium as time goes to infinity by means of a suitable {\L}ojasiewicz-Simon type inequality. A generalization and examples of applications will be given at the end of the paper.
\end{abstract}

\\

\section{Introduction}

In this paper, we are concerned some classes of nonlinear abstract damped wave equations, whose prototype is the usual wave equation in a bounded open domain $\mathbb{R}^{N}$, $N\geq1$,
\begin{eqnarray}
\left
    \{
        \begin{array}{ll}
            \ddot{u} + h(t)\dot{u} - \Delta(u) = f(u) &\textrm{in $\mathbb{R}^{+} \times \Omega$;} \\
                                             u = 0 &\textrm{on $\mathbb{R}^{+} \times\partial\Omega$;}\\
                                        u(0,x) = u_{0}(x) &\textrm{in $\Omega$;}\\
                                  \dot{u}(0,x) = u_{1}(x) &\textrm{in $\Omega$.} \label{1}
        \end{array}
\right.
\end{eqnarray}
where $h$, $f$ are suitably given.

This problem has been already investigated by many authors.
Concerning the damping $h$, different assumptions are alternatively made: on-off(\cite{1}, \cite{7}), increasing(\cite{12}), bounded or constant(\cite{13}, \cite{14}, \cite{2}, \cite{3}), integrally positive(\cite{1}, \cite{11}), etc. In particular, on-off dampers are suitable to describe a wide variety of communication network models, as well as systems where a control depending on time is necessary.

In the earlier papers, When $h(t)$ is a constant, there has been some results about the
asymptotics for the equations, such as M. Jendoubi \cite{8}, A.
Haraux and M. Jendoubi \cite{2}, \cite{3}.
Moreover, convergence to an equilibrium has been
established in many case, when the damping term $Q = g(\dot{u})$, which is linear or nonlinear without being dependent on time.
Especially, in the paper \cite{16}, the authors considered the general form. The key point is that all these papers used an inequality,
so-called \emph{{\L}ojasiewicz-Simon} inequality, to obtain their results. However, it must assume the nonlinearity $f(s)$ is
analytic with respect to $s$.

When the nonlinearity $f\equiv 0$, the papers \cite{7}, \cite{10} have
obtained the energy inequalities by doing research to the
damping term $Q = g(t, \dot{u})$ in detail. In \cite{10} where $Q =
\rho(t)g(\dot{u})$, the author give a classification of the
behaviors of the damping near the origin and at the infinity, and
introduce some auxiliary functions to determine the rate of decay of
the energy functional. However, In \cite{9} where $Q = g(t, \dot{u})$,
assumes growth conditions at infinity. Very interesting results in the special case $f\equiv 0$, and damping of type on-off
can be found in \cite{7}, also when the term $h(t)\dot{u}$ is replaced by $h(t)g(\dot{u})$, where $g$ is a nonlinear function with linear growth(see also \cite{9}). In addition, A logarithmic decay estimate is proved in \cite{15} when the term $h(t)\dot{u}$ is replaced by $(1+t)^{\theta}a(x)g(\dot{u})$, with $a$ bounded and positive on a subdomain of $\Omega$ and $g$ possibly having superlinear growth at infinity.

In \cite{11}, the author shows that, if $f$ has linear growth and $h$ is integrally positive, then any solution $u$ of problem (\ref{1}) converges to 0 in the norm $\|\bigtriangledown u\|_{L^{2}} + \|u_{t}\|_{L^{2}}$ if and only if
\begin{eqnarray*}
    \int_{0}^{\infty}e^{-H(t)}\int_{0}^{t}e^{H(s)}dsdt=+\infty,
\end{eqnarray*}
with $H(t)=\int_{0}^{t}h(s)ds$.

In this paper we prove global stability for problem (\ref{1}) when $h$ is integrally positive,  and also $f$ satisfies sign condition and regularity assumption. This result is interesting because $h$ does not need to satisfy any other condition, and no growth condition on $f$ is required. However our result is only applicable to strong solution and the condition of the trajectory of the solution is bounded in $W^{2, p}(\Omega) \times W^{1, p}(\Omega)$, with $p > \frac{N}{2}$ is restrictive, and not always easy to check in practice. For damping of type on-off, we make a special assumption( see Theorem 2), which is sufficient for the stability. Unfortunately, we still do not know if this assumption is also necessary for stability to hold.

The plan of this paper is as follows: In Section 2 we give some definitions and assumptions, and in Section 3 we state our main
result, and give the proof. And we will obtain the decay rates. Section 4 is devoted to a generalization of our results. In each section some remarks are presented.

\begin{acknowledgment}
    The author thanks the referees and Prof. Ti-Jun Xiao for their valuable suggestions concerning the presentation of our results.
\end{acknowledgment}

\section{Preliminaries}
Let us begin with the following definition and assumptions.
\begin{definition}
A function $h$: $[0, +\infty)\rightarrow [0, +\infty)$ is said \textit{integrally positive} if for every $\epsilon > 0$ there exist $\delta > 0$ such that
\[
    \int_{t}^{t+\epsilon}h(s)\mathrm{d}s \geq \delta,  \forall t > 0.
\]
\end{definition}
\begin{remark}
    We emphasize the fact that the function $h$ may vanish somewhere but not on any interval according to this definition.
    In addition, we can easily get that there exist a constant $\kappa > 0$ such that $h(t) > \kappa$ for almost every $t \in \mathbb{R}$.
\end{remark}

\begin{definition}
Assume there is a sequence $(J_{n})_{n\geq1}$ of disjoint open intervals in $(0, +\infty)$, $J_{n} = (a_{n}, b_{n})$, where $a_{1} = 0$, $b_{n} = a_{n+1}$ and $a_{n} \rightarrow +\infty$, if $h$: $[0, +\infty)\rightarrow \mathbb{R}$ satisfies: for all $t > J_{n}$, there exist $0 < m_{n} \leq M_{n} < \infty$ such that
\begin{eqnarray*}
    m_{n} \leq h(t) \leq M_{n},
\end{eqnarray*}
we call $h(t)$ is in \textit{the positive-negative} case.
\end{definition}

\begin{remark}
    Notice that this kind of intermitting damping may change sign at the discontinuous points. If all the discontinuous points $h(b_{n}) = 0$, we call this damping is in \textit{on-off} case.
\end{remark}

For the nonlinearity $f$, we note that the sign of $f$ looks quite important. Indeed, it is well known that solutions of $\ddot{u} + h(t)\dot{u} - \Delta(u) = u^{3}$, $h(t)\geq 0$, may blow up in finite time( see \cite{6}, \cite{17}).
\begin{assumption}[Sign assumption]
    Assume $f$ satisfies
    \[
        sf(s) \leq 0, \forall s \in \mathbb{R}
    \]
\end{assumption}

From the sign assumption, we can easily have
\[
    F(s)\triangleq \int_{0}^{s}f(\tau)\mathrm{d}\tau \leq 0, \forall s \in \mathbb{R}.
\]

\begin{assumption}[Regularity assumption]
    Assume $f$ satisfies
        \begin{itemize}
            \item $f(s)$ is analytic in $s$;
            \item $f$, $ f' $ and $ f'' $are bounded in $(-\beta, +\beta)$ for all $\beta > 0$;
        \end{itemize}
\end{assumption}

%\begin{remark}
%    In both assumptions, $f(u)=-|u|^{p}u$, $p\geq1$, is the prototype function(see, for example, \cite{6}, \cite{17}).
%\end{remark}

For each solution $u$ of problem (\ref{1}), we define its energy by
\begin{eqnarray}
    E_{u}(t) =  \int_{\Omega}\frac{1}{2}( |u_{t}|^{2} + |\nabla(u)|^{2} ) - F(u)\mathrm{d}x \label{2}
\end{eqnarray}
where $F(u)\triangleq \int_{0}^{u}f(s)\mathrm{d}s$. In addition, we denote
\begin{eqnarray}
    e_{u}(t)=\int_{\Omega}\frac{1}{2}|\nabla(u)|^{2}  - F(u)\mathrm{d}x. \label{3}
\end{eqnarray}
If there is no need to specify $u$, we simplify $E_{u}(t), e_{u}(t)$ by $E(t), e(t)$ respectively.

Finally, we use the letter $C$ below to denote corresponding constants, and also denote
\begin{eqnarray*}
    \mathcal{H} =  H_{0}^{1}(\Omega) \times L^{2}(\Omega)
\end{eqnarray*}
which is often referred to as the natural energy space.
\begin{eqnarray*}
    \mathcal{D} = \{ (u, v)^{T} \in H^{2}(\Omega) \times H^{1}(\Omega):  u\mid _{\textrm{$\partial\Omega$}} = 0 \}
\end{eqnarray*}
which is clearly a closed subspace of $H^{2}(\Omega) \times H^{1}(\Omega)$.

In our setting before we will obtain the following results. The proof is an adaptation to the one given therein and is thus omitted.

\begin{proposition}
    If $(u_{0}, u_{1})^{T} \in \mathcal{D}$,  Then problem (\ref{1}) has a unique solution $(u, u_{t})^{T}$, and we have
        \begin{eqnarray*}
                (u, u_{t})^{T} \in C([0, T]; \mathcal{D})\cap C^{1}([0, T]; \mathcal{H});
        \end{eqnarray*}
        And there exists a constant $C > 0$ such that
        \begin{eqnarray*}
                \| (u, u_{t})^{T} \|_{\mathcal{D}} \leq K,
        \end{eqnarray*}
        with $K$ only depending on $\| (u_{0}, u_{1})^{T} \|_{\mathcal{D}}$ .
\end{proposition}

\begin{remark}
    For the proof of this proposition, one can refer to \cite{18}. In \cite{18}, the authors proved these results for problem (\ref{1}) without the damping term. But the method of proving our results here is similar.
\end{remark}

\begin{proposition}
    For any solution $(u, u_{t})^{T}$ of problem (\ref{1}) we have
    \[
        E_{u}'(t) = - \int_{\Omega}h(t)u_{t}^{2}\mathrm{d}x,
    \]
    which is non-positive if $h \geq 0$.
\end{proposition}

\section{Main result}

Now we can state our first fundamental results.
\begin{main}
    Assume $h$ is integrally positive and $f$ satisfies Assumption 1 and 2. Let $(u, u_{t})^{T}$ be a solution of problem (\ref{1}) with $(u_{0}, u_{1})^{T} \in \mathcal{D}$, and there exists $p \geq 2$ such that its trajectory is precompact in $W^{2, p}(\Omega) \times W^{1, p}(\Omega)$, with $p > \frac{N}{2}$ if $N \leq 6$ and $p > N$ if $ N > 6$. Then there exists a equilibrium $(\varphi, 0)$, with $\varphi$ in the set
    \begin{eqnarray}
        \Sigma = \{ \varphi \in H^{2}(\Omega)\cap H^{1}_{0}(\Omega):  - \Delta \varphi = f(\varphi) \}, \label{4}
    \end{eqnarray}
    such that
    \begin{eqnarray*}
        \lim_{t \rightarrow \infty}{\|u_{t}\|_{L^{2}} + \|u - \varphi\|_{W^{2, p}}} = 0.
    \end{eqnarray*}
    And there exists $\theta = \theta(\varphi) \in (0, \frac{1}{2}]$ such that
    \begin{itemize}
        \item if $0 < \theta < \frac{1}{2}$, then
                \begin{eqnarray*}
                    \|u - \varphi\|_{H^{1}} + \|u_{t}\|_{L^{2}} = o(t^{-\frac{\theta}{1-2\theta}}), t \rightarrow \infty
                \end{eqnarray*}
        \item if $\theta = \frac{1}{2}$, then
                \begin{eqnarray*}
                    \|u - \varphi\|_{H^{1}} + \|u_{t}\|_{L^{2}} = o(e^{-\zeta t}), t \rightarrow \infty
                \end{eqnarray*}
                with $\zeta > 0$.
    \end{itemize}

\end{main}

\begin{remark}
    In fact, we can consider $f(x, s)$ which is analytic in $s$, uniformly with respect to $x$
\end{remark}

Before our proof of theorem 1, let us give some Lemmas.

\begin{lemma}
    Under Assumption 1, let $(u, u_{t})^{T}$ be a solution of problem (\ref{1}), we can obtain
        \begin{eqnarray}
            \lim_{t \rightarrow \infty}\|u_{t}\|_{L^{2}} = 0.  \label{5}
        \end{eqnarray}
\end{lemma}
\begin{proof}
    \textbf{Step 1}

    By Proposition in section 2, there exists $E_{\infty} \geq 0$ such that
    \begin{eqnarray}
        \lim_{t \rightarrow \infty}E_{u}(t) = E_{\infty}. \label{6}
    \end{eqnarray}
    By assumption 1, we know $\int_{\Omega}F(u)\mathrm{d}x \leq 0$, so that by (10) there exists $L \in [0, E_{\infty}]$ such that
    \begin{eqnarray}
        \limsup_{t \rightarrow \infty}\|u_{t}\|^{2}_{L^{2}} = L.  \label{7}
    \end{eqnarray}
    We want to show that $L = 0$, so let us assume by contradiction that $L > 0$.

    First, It is easy for us to get a important formula:
    \begin{eqnarray}
        0 < E_{\infty} &=& E(0) + \int_{0}^{\infty}E'(\tau)\mathrm{d}\tau \nonumber \\
                       &=& E(0) - \int_{0}^{\infty}h(\tau)\| u'(\tau) \|^{2}_{L^{2}}\mathrm{d}\tau.  \label{8}
    \end{eqnarray}
    Next, we must distinguish two cases.

    \emph{First case}: $\|u_{t}\|^{2}_{L^{2}} \equiv L$, $\forall t > 0$.

    Then by (\ref{8}) we get
    \begin{eqnarray}
       0 < E_{\infty} = E(0) - L \int_{0}^{\infty}h(\tau)\mathrm{d}\tau.  \label{9}
    \end{eqnarray}
    Since $h$ is integrally positive, there exists $\delta > 0$ such that
    \begin{eqnarray}
       \int_{n}^{n+1}h(s)\mathrm{d}s \geq \delta, \forall n \in \mathbb{N} .  \label{10}
    \end{eqnarray}
    Thus, (\ref{9}) and (\ref{10}) imply
    \[
       0 < E(0) - L \sum_{n=1}^{\infty} \delta = -\infty,
    \]
    which is impossible obviously.

    \emph{Second case}: $\|u_{t}\|^{2}_{L^{2}} \neq L$

    Then let us set
    \begin{eqnarray}
        \liminf_{t \rightarrow \infty}\|u_{t}\|^{2}_{L^{2}} = l \in [0, L). \label{11}
    \end{eqnarray}
    Since $u \in C^{1}([0, T]; L^{2})$ by Proposition 1 in section 2, (\ref{7}) and (\ref{11}) implies that there exist two sequences $(t_{n})_{n}$ and $(\overline{t}_{n})_{n}$ such that
    \begin{itemize}
        \item $t_{n} \rightarrow +\infty$, as $ n \rightarrow \infty$;
        \item $0< t_{n} < \overline{t}_{n} < t_{n}$, $\forall n \in \mathbb{N}$;
        \item $\frac{L+l}{2} = \|u'(t_{n})\|_{L^{2}} < \|u'(\overline{t}_{n})\|_{L^{2}} = \frac{3L+l}{4}$, $\forall n \in \mathbb{N}$;
        \item $\frac{L+l}{2} \leq\|u'(t)\|_{L^{2}} \leq \frac{3L+l}{4}$, $\forall t \in (t_{n}, \overline{t}_{n})$.
    \end{itemize}

    By proposition 1, there exists $K>0$ such that
    \begin{eqnarray}
        \frac{d}{dt}\|u'(t)\|^{2}_{L^{2}} &=& 2\langle u'(t), u''(t) \rangle \nonumber  \\
                                          &=& 2\langle u'(t), \triangle u - h(t)u'(t) + f(u)\rangle \nonumber \\
                                          &\leq & 2\|u'(t)\|_{L^{2}} \|\triangle u\|_{L^{2}} + 2\|u'(t)\|_{L^{2}}\|f(u)\|_{L^{2}} \nonumber \\
                                          &\leq &  K.   \label{12}
    \end{eqnarray}
    By integrating (\ref{12}) over $(t_{n}, \overline{t}_{n})$ we get
    \begin{eqnarray*}
        K(\overline{t}_{n} - t_{n}) &\geq& \int_{t_{n}}^{\overline{t}_{n}}\frac{d}{dt}\|u'(t)\|^{2}_{L^{2}}\mathrm{d}t  \\
                                     &=& \|u'(\overline{t}_{n})\|_{L^{2}} - \|u'(t_{n})\|_{L^{2}}  \\
                                     &=&  \frac{3L+l}{4} - \frac{L+l}{2} \\
                                     &=& \frac{L-l}{4}.
    \end{eqnarray*}
    So
    \begin{eqnarray}
        \overline{t}_{n} - t_{n} \geq \frac{L-l}{4K}, \forall n \in \mathbb{N}.  \label{13}
    \end{eqnarray}

    By (\ref{8}) and (\ref{13}) we get
    \begin{eqnarray*}
       0 < E_{\infty} \leq E(0) -  \int_{\cup(t_{n}, t_{n} + \frac{L-l}{4K})}h(\tau)\| u'(\tau) \|^{2}_{L^{2}}\mathrm{d}\tau.
    \end{eqnarray*}
    Since $\frac{L+l}{2} \leq\|u'(t)\|_{L^{2}} $, $\forall t \in (t_{n}, t_{n} + \frac{L-l}{4K})$ and $h$ is integrally positive, i.e. there exists $\delta > 0$ such that
    \begin{eqnarray}
       \int_{(t_{n}, t_{n} + \frac{L-l}{4K})}h(s)\mathrm{d}s \geq \delta, \forall n \in \mathbb{N},  \label{14}
    \end{eqnarray}
    so that
    \[
       0 < E(0) - \frac{L+l}{2} \sum_{n=1}^{\infty} \delta = -\infty,
    \]
    and a contradiction arises. Furthermore, we obtain
    \begin{eqnarray}
        \liminf_{t \rightarrow \infty}\|u_{t}\|^{2}_{L^{2}} = L.  \label{15}
    \end{eqnarray}

    (\ref{7}) and (\ref{15}) imply
    \begin{eqnarray}
        \lim_{t \rightarrow \infty}\|u_{t}\|^{2}_{L^{2}} = L.  \label{16}
    \end{eqnarray}

    \textbf{Step 2}

    In this step we will proof $L=0$.

    By (\ref{16}), there exists $T > 0$ such that
    \begin{eqnarray}
        \|u_{t}\|^{2}_{L^{2}} \geq \frac{L}{2}, \forall t \geq T.  \label{17}
    \end{eqnarray}
    By (\ref{8}) and (\ref{17}), we get
    \begin{eqnarray}
       0 < E_{\infty} \leq E(0) - \frac{L}{2} \int_{T}^{\infty}h(\tau)\mathrm{d}\tau. \label{18}
    \end{eqnarray}
    Since $h$ is integrally positive, there exists $\delta > 0$ such that
    \begin{eqnarray*}
       \int_{n}^{n+1}h(s)\mathrm{d}s \geq \delta, \forall n \in \mathbb{N} .
    \end{eqnarray*}
    Thus
    \[
       0 < E(0) - \frac{L}{2} \sum_{n=1}^{\infty} \delta = -\infty,
    \]
    again a contradiction.

    Above all, we can conclude that $L = 0$.

\end{proof}

\begin{remark}
    In proving the previous result actually we can also know
    \[
        \lim_{t \rightarrow \infty}e_{u}(t) = E_{\infty}.
    \]
\end{remark}

For the proof of the main theorem, we have to use the following generalization of the {\L}ojasiewicz-Simon inequality established in \cite{4}, see also \cite{5} for a previous variant.

\begin{lemma}
    Under Assumption 2, and let $\varphi \in \sum$, then there exist $\theta \in (0, \frac{1}{2}]$ and $\delta > 0$ such that
    $\forall u\in H^{1}_{0}(\Omega)$, $ \|u - \varphi\|_{W^{2, p}} < \delta$,
        \begin{eqnarray}
            \|\Delta u + f(u)\|_{L^{2}} \geq \mid e_{u} - e_{\varphi}\mid^{1-\theta}  \label{19}
        \end{eqnarray}
\end{lemma}

To estimate the rate of decay of the difference between a solution and equilibrium, we have to use the following Lemma from \cite{3}.

\begin{lemma}
    Let $T > 0$, $\upsilon \in W^{1, 1}$, $\upsilon \geq 0$ in $[0, T]$ such that:
    \begin{eqnarray*}
        \upsilon'(t) \leq -C[\upsilon(t)]^{\alpha},  a.e. in [0, T],
    \end{eqnarray*}
    where $C$ and $\alpha$ are two constants. Then
    \begin{itemize}
        \item if $1 < \alpha $, then we have, with $\beta = \frac{1}{\alpha - 1}$ and $C' = [C(\alpha - 1)]^{-\beta}$, the inequality
                \begin{eqnarray}
                    \upsilon(t) \leq C't^{-\beta}, t \in [0, T];  \label{20}
                \end{eqnarray}
        \item if $\alpha = 1$, then
                \begin{eqnarray}
                    \upsilon(t) \leq \upsilon(0)e^{-Ct}, t \in [0, T].  \label{21}
                \end{eqnarray}
    \end{itemize}
\end{lemma}

Now let us begin our proof of theorem 1:

\begin{proof}[Proof of Theorem 1]

\textbf{Step 1}

Let us also define the $\omega-limit$ set of $(u_{0}, u_{1}) \in \mathcal{H}$ by
    \begin{eqnarray}
    \omega(u_{0}, u_{1})&=& \{ (\varphi, 0)| \varphi \in H^{1}_{0}(\Omega)\cap W^{2, p}(\Omega), \nonumber  \\
    && \exists (t_{n})_{n\geq1}:  t_{n}\rightarrow \infty,  s. t. \lim_{n \rightarrow \infty}\|u(t_{n}, x) - \varphi(x) \|_{W^{2, p}} = 0\} \label{22}
    \end{eqnarray}
where $u$ is the global solution of (1). Then we have
\begin{itemize}
    \item $\omega(u_{0}, u_{1})$ is nonempty, compact and connected set;
    \item $\forall(\varphi, 0)\in \omega(u_{0}, u_{1})$ we have $ -\bigtriangleup \varphi = f(\varphi)$, i.e. $\omega(u_{0}, u_{1}) \in \Sigma$;
    \item $e_{u}(t)$ is constant over $\omega(u_{0}, u_{1})$.
\end{itemize}
We note that $\omega-limit$ set is a subset of the set of stationary solutions.

\textbf{Step 2}

Without loss of generality, we assume $h(t) > \kappa$ for all $t \in \mathbb{R}$. Now let $\varepsilon < \kappa $ be a positive real number, and we define the Lyapounov functional
\begin{eqnarray}
    H(t) &=& E_{u}(t) - \varepsilon^{2} \int_{\Omega}( \triangle u +f(u)) u_{t}\mathrm{d}x + \varepsilon \int_{\Omega}|\nabla u_{t}|^{2}\mathrm{d}x \nonumber \\
        && + \varepsilon \int_{\Omega}| \triangle u +f(u)|^{2}\mathrm{d}x - \varepsilon \int_{\Omega} f'(u) |u_{t}|^{2}\mathrm{d}x,  \label{23}
\end{eqnarray}
for all $t > 0$. We note that $H$ makes sense as a consequence of our assumption on the trajectory.

\emph{2.1} Estimation on $H'(t)$

We can easily have:
\begin{eqnarray*}
    H'(t) &=& -h(t)\int_{\Omega}| u_{t}|^{2}\mathrm{d}x - \varepsilon^{2} \int_{\Omega}( \triangle u +f(u)) (-h(t)u_{t}+\triangle u
             \\
         && +f(u))\mathrm{d}x - \varepsilon^{2} \int_{\Omega}( \triangle u_{t} +f'(u)u_{t}) u_{t}\mathrm{d}x - 2\varepsilon \int_{\Omega}\triangle u_{t} u_{tt}\mathrm{d}x
             \\
         && + 2\varepsilon \int_{\Omega}( u_{tt} + h(t)u_{t})(\triangle u_{t} +f'(u)u_{t})\mathrm{d}x
             \\
         && - \varepsilon \int_{\Omega} f''(u)u_{t}|u_{t}|^{2}\mathrm{d}x - 2\varepsilon \int_{\Omega} f'(u) u_{t} u_{tt}\mathrm{d}x
             \\
         &=&  \int_{\Omega}[-h(t) + (2 \varepsilon h(t)-\varepsilon^{2} )f'(u)]| u_{t}|^{2}\mathrm{d}x   \\
         && - \varepsilon \int_{\Omega} f''(u)u_{t}|u_{t}|^{2}\mathrm{d}x - \varepsilon^{2} \int_{\Omega}| \triangle u +f(u)|^{2}\mathrm{d}x    \\
         && + \varepsilon^{2} h(t) \int_{\Omega}( \triangle u +f(u))u_{t}\mathrm{d}x -(2 \varepsilon h(t)-\varepsilon^{2})\int_{\Omega}|\nabla u_{t}|^{2}\mathrm{d}x
\end{eqnarray*}
Using the $\varepsilon$-Young inequality in the term $ \int_{\Omega}( \triangle u +f(u))u_{t}\mathrm{d}x $, we find
\begin{eqnarray*}
    H'(t) &\leq& \int_{\Omega}[-\frac{h(t)}{2} + (2 \varepsilon h(t)-\varepsilon^{2} )f'(u)]| u_{t}|^{2}\mathrm{d}x \mathrm{d}x     \\
             &&  - \varepsilon \int_{\Omega}f''(u)u_{t}|u_{t}|^{2} - (\varepsilon^{2} - \frac{\varepsilon^{4}}{2} ) \int_{\Omega}| \triangle u
                 +f(u)|^{2}\mathrm{d}x  \\
             &&  -(2 \varepsilon h(t)-\varepsilon^{2})\int_{\Omega}|\nabla u_{t}|^{2}\mathrm{d}x \\
          &\leq& \int_{\Omega}[(-\frac{1}{2} + 2 \varepsilon f'(u))h(t)-\varepsilon^{2} f'(u)]| u_{t}|^{2}\mathrm{d}x \mathrm{d}x     \\
             &&  - \varepsilon \int_{\Omega}f''(u)u_{t}|u_{t}|^{2} - \frac{\varepsilon^{2}}{2} \int_{\Omega}| \triangle u
                 +f(u)|^{2}\mathrm{d}x  \\
             &&  -\varepsilon^{2}\int_{\Omega}|\nabla u_{t}|^{2}\mathrm{d}x
\end{eqnarray*}
The caculation is formal, but can be rigorously justified by using our assumption on the trajectory. Moreover, we note that $u \in W^{2, p}(\Omega) \hookrightarrow L^{\infty}(\Omega)$. Then by the regularity assumption on $f$ we know $f'(u)$ and $f''(u)$ remain bounded.

When $N \leq 6$, we have $ H_{0}^{1}(\Omega) \hookrightarrow L^{3}(\Omega)$ with continuous embedding and then we obtain
\begin{eqnarray*}
    - \varepsilon \int_{\Omega}f''(u)u_{t}|u_{t}|^{2}\mathrm{d}x &-& \varepsilon^{2}\int_{\Omega}|\nabla u_{t}|^{2}\mathrm{d}x \\
    &\leq& \varepsilon C \|u_{t}\|^{3}_{L^{3}} -\varepsilon^{2}\int_{\Omega}|\nabla u_{t}|^{2}\mathrm{d}x \\
    &\leq& \varepsilon C \|u_{t}\|^{3}_{H_{0}^{1}} -\varepsilon^{2}\int_{\Omega}|\nabla u_{t}|^{2}\mathrm{d}x.
\end{eqnarray*}
By lemma 1, we know $\lim\|u_{t}\|_{L^{2}} = 0$ and then there exists $T_{1} > 0$ such that for all $t > T_{1}$
\begin{eqnarray*}
    - \varepsilon \int_{\Omega}f''(u)u_{t}|u_{t}|^{2}\mathrm{d}x &-& \varepsilon^{2}\int_{\Omega}|\nabla u_{t}|^{2}\mathrm{d}x
        \\
    &\leq& \varepsilon C \|u_{t}\|^{2}_{L^{2}} - \frac{1}{2}\varepsilon^{2}\int_{\Omega}|\nabla u_{t}|^{2}\mathrm{d}x.
\end{eqnarray*}

When $N > 6$, we have $ u_{t} \in W_{1}^{P}(\Omega) \hookrightarrow L^{\infty}(\Omega)$. So
\begin{eqnarray*}
    - \varepsilon \int_{\Omega}f''(u)u_{t}|u_{t}|^{2}\mathrm{d}x &-& \varepsilon^{2}\int_{\Omega}|\nabla u_{t}|^{2}\mathrm{d}x
        \\
    &\leq& \varepsilon C \|u_{t}\|^{2}_{L^{2}} -\varepsilon^{2}\int_{\Omega}|\nabla u_{t}|^{2}\mathrm{d}x.
\end{eqnarray*}

In both cases, by choosing $\varepsilon$ small enough there exists $C > 0$, which is independent on $t$, such that for all $t > T_{1}$
\begin{eqnarray}
    -H'(t) &\geq &  C \{ \|u_{t}\|^{2}_{H^{1}} +  \| \triangle u +f(u) \|^{2}_{L^{2}} + \|\nabla u_{t}\|^{2}_{L^{2}} \} \nonumber \\
           &\geq & \frac{C}{3}\{ \|u_{t}\|_{H^{1}} +  \| \triangle u +f(u) \|_{L^{2}} \} ^{2},  \label{24}
\end{eqnarray}
and $ H(t) \geq 0$.

Then $H$ is non-increasing  on $[T_{1}, +\infty)$,  and so that $H(t)$ has a limit at infinity.

Since $(\varphi, 0)\in \omega(u_{0}, u_{1})$, there exists $(t_{n})_{n\geq1}:  t_{n}\rightarrow \infty$, such that
\begin{eqnarray}
    u(t_{n}, x) \xrightarrow{W^{2, p}} \varphi(x),    n \rightarrow \infty.  \label{25}
\end{eqnarray}
Moreover, we also get
\begin{eqnarray}
    \lim_{n \rightarrow \infty}e_{u}(t_{n}) = e_{\varphi}.  \label{26}
\end{eqnarray}

\emph{2.2} Estimation on $[H(t)-e_{\varphi}]^{1-\theta}$

Let $\theta \in (0, \frac{1}{2}]$ be as in Lemma 2, then by using Holder's inequality we get
\begin{eqnarray}
  [H(t)-e_{\varphi}]^{1-\theta} &\leq& C \{ \|u_{t}\|^{2(1-\theta)}_{L^{2}} + |e_{u}(t)-e_{\varphi}|^{1-\theta}   \nonumber \\
                    &&  + \| \triangle u +f(u) \|^{2(1-\theta)}_{L^{2}}+\|\nabla u_{t}\|^{2(1-\theta)} \nonumber \\
                    &&  + \| \triangle u +f(u) \|^{(1-\theta)}_{L^{2}} \|u_{t}\|^{(1-\theta)}_{L^{2}}\}.  \label{27}
\end{eqnarray}
Thanks to Young's inequality we have
\begin{eqnarray}
  \| \triangle u +f(u) \|^{(1-\theta)}_{L^{2}} \|u_{t}\|^{(1-\theta)}_{L^{2}} &\leq& (1-\theta)\| \triangle u +f(u) \|_{L^{2}} + \theta \|u_{t}\|^{\frac{(1-\theta)}{\theta}}_{L^{2}}\nonumber \\
  & \leq & \| \triangle u +f(u) \|_{L^{2}} +  \|u_{t}\|^{\frac{(1-\theta)}{\theta}}_{L^{2}}. \nonumber
\end{eqnarray}
Then (\ref{27}) becomes
\begin{eqnarray}
  [H(t)-e_{\varphi}]^{1-\theta} &\leq& C \{ \|u_{t}\|^{2(1-\theta)}_{L^{2}} + |e_{u}(t)-e_{\varphi}|^{1-\theta} \nonumber  \\
                    && + \| \triangle u +f(u) \|^{2(1-\theta)}_{L^{2}}+\|\nabla u_{t}\|^{2(1-\theta)} \nonumber \\
                     && +  \| \triangle u +f(u) \|_{L^{2}} +  \|u_{t}\|^{\frac{(1-\theta)}{\theta}}_{L^{2}}\}.  \label{28}
\end{eqnarray}
 Note that $2(1-\theta)>1$ and $\frac{(1-\theta)}{\theta} > 1$,  and by (\ref{5}) we know there exists $T_{2} > T_{1}$ such that for all $t > T_{2}$
\begin{eqnarray}
  [H(t)-e_{\varphi}]^{1-\theta} &\leq& C \{ \|u_{t}\|_{H^{1}} + |e_{u}(t)-e_{\varphi}|^{1-\theta} \nonumber \\
                    &&     + \| \triangle u +f(u) \|_{L^{2}} \}.  \label{29}
\end{eqnarray}

\textbf{Step 3}

Since $H$ has a limits at infinity and By (\ref{25}), we have for all $\sigma > 0$, $\sigma \ll \delta$ there exists $N > 0$ such that $t_{N} \geq T_{2}$ and
\begin{eqnarray}
    \|u(t_{N}, x) - \varphi(x) \|_{W^{2, p}} < \frac{\sigma}{2} ,  \label{30}
\end{eqnarray}
\begin{eqnarray}
    \frac{C}{\theta}\{ [H(t_{N})-e_{\varphi}]^{\theta} - [H(t)-e_{\varphi}]^{\theta}\} < \frac{\sigma}{2},  \label{31}
\end{eqnarray}
\begin{eqnarray}
    H(t) \geq e_{\varphi},  \label{32}
\end{eqnarray}
for all $ t \geq t_{N}$.

Let
\begin{eqnarray}
    \overline{t} = Sup\{ t \geq t_{N}:  \|u(s, x) - \varphi(x) \|_{W^{2, p}} < \delta, \forall s \in [t_{N}, t]\}  \label{33}
\end{eqnarray}
Then by Lemma 2 and (\ref{29}) we have for all $t \in [t_{N}, \overline{t})$
\begin{eqnarray}
  [H(t)-e_{\varphi}]^{1-\theta} \leq 2C \{ \|u_{t}\|_{H^{1}} + \| \triangle u +f(u) \|_{L^{2}} \}.  \label{34}
\end{eqnarray}

Since we know
\begin{eqnarray}
    -\frac{\mathrm{d}}{\mathrm{d}t} [H(t)-e_{\varphi}]^{\theta} = -\theta[H(t)-e_{\varphi}]^{\theta - 1}H'(t) ,  \label{35}
\end{eqnarray}
by (\ref{24}) and (\ref{34}) we have
\begin{eqnarray}
    -\frac{\mathrm{d}}{\mathrm{d}t}[H(t)-e_{\varphi}]^{\theta} \geq \theta C \{ \|u_{t}\|_{H^{1}} + \| \triangle u +f(u) \|_{L^{2}} \}. \label{36}
\end{eqnarray}
By integrating (\ref{36}) over $[t_{N}, \overline{t})$ we get
\begin{eqnarray}
     \int_{t_{N}}^{\overline{t}}\|u_{t}\|_{H^{1}}\mathrm{d}t &\leq& \int_{t_{N}}^{\overline{t}}\{ \|u_{t}\|_{H^{1}} + \| \triangle u +f(u) \|_{L^{2}} \}\mathrm{d}t \nonumber \\
     &\leq& \frac{C}{\theta} \{[H(t_{N})-e_{\varphi}]^{\theta} - [H(t)-e_{\varphi}]^{\theta} \}  .   \label{37}
\end{eqnarray}

Assuming $\overline{t}<\infty$, we have
\begin{eqnarray*}
   \|u(\overline{t}, x) - \varphi(x) \|_{H^{1}}  &\leq&  \|u(t_{N}, x) - \varphi(x) \|_{H^{1}}  + \int_{t_{N}}^{\overline{t}}\|u_{t}\|_{H^{1}}\mathrm{d}t \\
   &\lesssim &  \|u(t_{N}, x) - \varphi(x) \|_{W^{2, p}}  + \int_{t_{N}}^{\overline{t}}\|u_{t}\|_{H^{1}}\mathrm{d}t.
\end{eqnarray*}
By (\ref{30}), (\ref{31}) and (\ref{37}), we have
\begin{eqnarray*}
    \|u(\overline{t}, x) - \varphi(x) \|_{H^{1}}  \leq  \sigma,
\end{eqnarray*}
which contradicts (\ref{33}). Therefor $\overline{t}=\infty$.

Then by (\ref{37}) we have
\begin{eqnarray*}
     \int_{t_{N}}^{\infty}\|u_{t}\|_{H^{1}}\mathrm{d}t \leq +\infty,
\end{eqnarray*}
which implies the convergence of $u$ in $H^{1}$. Since the assumption on trajectory, we have
\begin{eqnarray}
    \lim_{t \rightarrow \infty}\|u(t, x) - \varphi(x)\|_{W^{2, p}} = 0.  \label{38}
\end{eqnarray}

\textbf{Step 4}

By (\ref{34}), (\ref{35}) and (\ref{36}), we have
\begin{eqnarray}
    \frac{\mathrm{d}}{\mathrm{d}t} [H(t)-e_{\varphi}] + C [H(t)-e_{\varphi}]^{2(1-\theta)} \leq 0,  \label{39}
\end{eqnarray}
for all $t \geq T = t_{N}$.

We can then apply Lemma 3. We have to distinguish 2 cases:
\\
\emph{Case 1}: $0 < \theta <\frac{1}{2} \Rightarrow  1 < 2(1 - \theta) < 2$  \\
By Lemma 3, we have for all $t \geq T$
\begin{eqnarray*}
    H(t)-e_{\varphi} \leq Ct^{-\frac{1}{1-2\theta}}.
\end{eqnarray*}
By integrating (\ref{36}) over $(t, \infty)$, $t \geq T$ we get
\begin{eqnarray}
     \int_{t}^{\infty} \{ \|u_{t}\|_{H^{1}} + \| \triangle u +f(u) \|_{L^{2}} \} \mathrm{d}\tau &\leq& \frac{C}{\theta} \{[H(t_{N})-e_{\varphi}]^{\theta} \nonumber \\
     && - [H(t)-e_{\varphi}]^{\theta} \} \nonumber \\
     &\leq &  Ct^{-\frac{\theta}{1-2\theta}}.  \label{40}
\end{eqnarray}
Furthermore, we have
\begin{eqnarray}
     \int_{t}^{\infty} \|u_{t}\|_{H^{1}}\mathrm{d}\tau  \leq   Ct^{-\frac{\theta}{1-2\theta}}.  \label{41}
\end{eqnarray}
\begin{eqnarray}
     \int_{t}^{\infty} \|u_{tt}\|_{L^{2}}\mathrm{d}\tau &\lesssim& \int_{t}^{\infty} \|u_{t}\|_{L^{2}}\mathrm{d}\tau +\int_{t}^{\infty} \| \triangle u +f(u) \|_{L^{2}}\mathrm{d}\tau \nonumber \\
     &\leq &  Ct^{-\frac{\theta}{1-2\theta}}.  \label{42}
\end{eqnarray}
By (\ref{41}) and (\ref{42}) we get
\begin{eqnarray*}
    \|u - \varphi\|_{H^{1}} + \|u_{t}\|_{L^{2}} \leq   Ct^{-\frac{\theta}{1-2\theta}}.
\end{eqnarray*}
\\
\emph{Case2}: $ \theta = \frac{1}{2} \Rightarrow  2(1 - \theta) = 1$ \\
By Lemma 3, we have for all $t \geq T$
\begin{eqnarray*}
    H(t)-e_{\varphi} \leq C\textrm{exp}(-Ct).
\end{eqnarray*}
By integrating (\ref{36}) over $(t, \infty)$, $t \geq T$ we get
\begin{eqnarray*}
     \int_{t}^{\infty} \{ \|u_{t}\|_{H^{1}} + \| \triangle u +f(u) \|_{L^{2}} \} \mathrm{d}\tau
     \leq   C\textrm{exp}(-Ct).
\end{eqnarray*}
Furthermore, we have
\begin{eqnarray}
     \int_{t}^{\infty} \|u_{t}\|_{H^{1}}\mathrm{d}\tau  \leq   C\textrm{exp}(-Ct).  \label{43}
\end{eqnarray}
\begin{eqnarray}
     \int_{t}^{\infty} \|u_{tt}\|_{H^{1}}\mathrm{d}\tau \leq   C\textrm{exp}(-Ct),  \label{44}
\end{eqnarray}
for $t \geq T$. By (\ref{43}) and (\ref{44}), and let $C$ a little bigger, then we get
\begin{eqnarray*}
    \|u - \varphi\|_{H^{1}} + \|u_{t}\|_{L^{2}} \leq   C\textrm{exp}(-Ct), t > 0.
\end{eqnarray*}

The theorem is completely proved.

\end{proof}

Using the same method, we can get another convergence to equilibrium theorem.

\begin{main}
    Assume $f$ satisfies Assumption 1 and 2, $h$ is in the positive-negative case. %with
%    \begin{eqnarray}
%        \sum_{n=1}^{\infty}m_{n}(b_{n}-a_{n}) = +\infty.  \label{45}
%    \end{eqnarray}
    Let $(u, u_{t})^{T}$ be a solution of problem (1), and there exists $p \geq 2$ such that its trajectory is precompact in $W^{2, p}(\Omega) \times W^{1, p}(\Omega)$, with $p > \frac{N}{2}$ if $N \leq 6$ and $p > N$ if $ N > 6$. Then there exists a equilibrium $\psi$ in the set
    \begin{eqnarray}
        \Sigma = \{ \psi \in H^{2}(\Omega)\cap H^{1}_{0}(\Omega):  - \Delta \psi = f(\psi) \}, \label{46}
    \end{eqnarray}
    such that
    \begin{eqnarray*}
        \lim_{t \rightarrow \infty}{\|u_{t}\|_{L^{2}} + \|u - \psi\|_{W^{2, p}}} = 0.
    \end{eqnarray*}
    And there exists $\theta = \theta(\psi) \in (0, \frac{1}{2}]$ such that
    \begin{itemize}
        \item if $0 < \theta < \frac{1}{2}$, then
                \begin{eqnarray*}
                    \|u - \psi\|_{H^{1}} = o(t^{-\frac{\theta}{1-2\theta}}), t \rightarrow \infty
                \end{eqnarray*}
        \item if $\theta = \frac{1}{2}$, then
                \begin{eqnarray*}
                    \|u - \psi\|_{H^{1}} = o(e^{-\xi t}), t \rightarrow \infty
                \end{eqnarray*}
                with $\xi > 0$ .
    \end{itemize}
\end{main}

%\begin{remark}
%    Under condition (\ref{45}), the distribution of the intervals has no importance. Only their size is important.
%\end{remark}

\begin{remark}
    Although our results derive the decay rates, we still don't know whether the rates are optimal.
\end{remark}

\section{Generalization and applications}

\subsection{Abstract damped wave equation}

Let $\Omega$ be a bounded open in $\mathbb{R}^{N}$, $N\geq1$ and let us
consider the following evolution problem:

\begin{eqnarray}
\left
    \{
        \begin{array}{ll}
            \ddot{u} + h(t)B\dot{u} + Au = f(u) &t > 0 \\
                                        u(0,x) = u_{0}(x)   &\\
                                  \dot{u}(0,x) = u_{1}(x)  &
        \end{array}
\right.   \label{47}
\end{eqnarray}
Here $H \triangleq L^{2}(\Omega)$ the usual Hilbert space with scalar product $\langle \cdot, \cdot\rangle_{H}$ and norm $\|\cdot \|_{H}$, $A: H \rightarrow H$ is a second order strongly elliptic operator on $H$ with dense domain, $D(A)\subset H$,
\begin{eqnarray*}
    D(A) = \{ \upsilon \in H: A\upsilon \in H, \upsilon \mid _{\textrm{$\partial\Omega$}} = 0   \},
\end{eqnarray*}
and $V = D(A^{\frac{1}{2}})$ with norm $\|\upsilon \|_{V} = \|A^{\frac{1}{2}}\upsilon \|_{H}$ is such that
\[
    V \hookrightarrow H \equiv H' \hookrightarrow V'
\]
with dense embeddings. Define $W= A^{-1}(H)$.

Let $B: H\rightarrow H$ be a bounded linear operator satisfying the coerciveness condition for all $\omega \in H$
\begin{eqnarray}
    a\|\omega\|_{H} \leq \langle B\omega, \omega \rangle_{H}.  \label{48}
\end{eqnarray}
for some $a>0$.

We assume the problem (\ref{47}) is variational, i.e. there exists a real-valued functional $\mathcal{F}$ such that $\mathcal{F}(0) = 0$ and $\mathcal{F}'(u)(\omega) = \langle f(u), \omega \rangle_{V', V}$ for any $u, \omega \in V$. Moreover we assume
\begin{eqnarray}
    \langle f(u), u \rangle_{V', V} \leq 0, \forall u \in V.  \label{49}
\end{eqnarray}
In addition, $f$ satisfies the assumption 2 in section 2.

Finally, we assume $h$ is integrally positive or in the positive-negative case.

For each solution $u$ of problem (\ref{47}), we define its energy by
\begin{eqnarray*}
    E_{u}(t) =  \frac{1}{2}( \|u_{t}\|^{2}_{H} + \|u(t)\|^{2}_{V} ) - \mathcal{F}(u(t)).
\end{eqnarray*}
In addition, we denote
\begin{eqnarray*}
    e_{u}(t)= \frac{1}{2}\|u(t)\|^{2}_{V} - \mathcal{F}(u(t)).
\end{eqnarray*}

Accordingly, we also have
\begin{proposition}
    For any solution $u$ of problem (\ref{47}) we have
        \begin{itemize}
            \item $(u, u_{t}) \in C([0, T]; W \times V)\cap C^{1}([0, T]; V \times H)$;
            \item $E_{u}'(t) = - \langle h(t)B\dot{u}, \dot{u} \rangle_{H}$, a.e. in $ [0, \infty)$, which is non-positive;
        \end{itemize}
\end{proposition}

Under these hypotheses, we have the following result:
\begin{main}
    Let $(u, u_{t})$ be a solution of problem (\ref{47}), and assume its trajectory is precompact in $W \times V$. Then there exists a equilibrium $\psi$ in the set
    \begin{eqnarray}
        \Sigma = \{ \psi \in D(A):  A \psi = f(\psi) \},  \label{50}
    \end{eqnarray}
    such that
    \begin{eqnarray}
        \lim_{t \rightarrow \infty}{\|u_{t}\|_{H} + \|u - \psi\|_{W}} = 0.  \label{51}
    \end{eqnarray}
\end{main}

\begin{remark}
    Let us observe that in the case $ H = L^{2}(\Omega)$, $A = -\triangle$, and $B = Id$. Then $W = H^{2}$,$V = H^{1}$, so that we can imply theorem 1 or 2 when $p=2$.
\end{remark}

The proof of Theorem 3 is the same as Theorem 1 and 2, but we have to use the following two Lemmas instead of Lemma 1 and 2.
\begin{lemma}
    Let $(u, u_{t})$ be a solution of problem (\ref{47}), we can obtain
        \begin{eqnarray}
            \lim_{t \rightarrow \infty}\|u_{t}\|_{H} = 0.  \label{52}
        \end{eqnarray}
\end{lemma}
\begin{lemma}
    Let $\psi \in \sum$, then there exist $\theta \in (0, \frac{1}{2}]$ and $\delta > 0$ such that
    $\forall u\in W$, $ \|u - \psi\|_{W} < \delta$,
        \begin{eqnarray}
            \|-A u + f(u)\|_{H} \geq \mid e_{u} - e_{\psi}\mid^{1-\theta}.  \label{53}
        \end{eqnarray}
\end{lemma}

We refer to \cite{5} for the proof of Lemma 5.

\subsection{Nonlinear interior damping}

We are concerned some classes of nonlinear damped wave equations in a bounded open domain $\mathbb{R}^{N}$, $N\geq1$,
\begin{eqnarray}
\left
    \{
        \begin{array}{ll}
            \ddot{u} + h(t)g(\dot{u}) - \Delta(u) = f(u) &\textrm{in $\mathbb{R}^{+} \times \Omega$;} \\
                                             u = 0 &\textrm{on $\mathbb{R}^{+} \times\partial\Omega$;}\\
                                        u(0,x) = u_{0}(x) &\textrm{in $\Omega$;}\\
                                  \dot{u}(0,x) = u_{1}(x) &\textrm{in $\Omega$.}
        \end{array}
\right.  \label{54}
\end{eqnarray}
Here $f$ satisfies Assumption 1 and 2, and $g$ are nonlinear functions subject to the following assumption.
    \begin{itemize}
        \item[(g-1)] $g \in C^{1}(\mathbb{R})$, $g$ is monotone increasing, and such that $0 < m_{1} \leq g'(s) \leq m_{2} <\infty$ for all $s \in \mathbb{R}$.
        \item[(g-2)] $g(0) = 0$.
    \end{itemize}

First we have a brief look at the nonlinear function $g$. For $s \geq 0$ we have
\[
   m_{2}s \geq g(s) = \int_{0}^{s}g'(\tau)\mathrm{d}\tau = s \int_{0}^{1}g'(\tau s)\mathrm{d}\tau \geq m_{1}s.
\]
Similarly, for $s < 0$ we obtain $m_{2}s \leq g(s) \leq m_{1}s $. These two formulas combined result in $ m_{1}s^{2} \leq g(s)s \leq m_{2}s^{2}$.

As to the energy functional, we have for any solution $(u, u_{t})^{T}$ of Eqs.(\ref{54}) we have
    \[
        E_{u}'(t) = - \int_{\Omega}h(t)g(u_{t})u_{t}\mathrm{d}x,
    \]
    which is non-positive if $h \geq 0$.

In this case Theorem 1 can be easily generalized as follows.
\begin{main}
    Assume $h$ is integrally positive or in the positive-negative case and $f$ satisfies Assumption 1 and 2. Let $(u, u_{t})^{T}$ be a solution of problem (\ref{54}) with $(u_{0}, u_{1})^{T} \in \mathcal{D}$, and its trajectory is precompact in $H^{2}(\Omega) \times H^{1}(\Omega)$. Then there exists a equilibrium $(\varphi, 0)$, with $\varphi$ in the set
    \begin{eqnarray*}
        \Sigma = \{ \varphi \in H^{2}(\Omega)\cap H^{1}_{0}(\Omega):  - \Delta \varphi = f(\varphi) \},
    \end{eqnarray*}
    and $\theta = \theta(\varphi) \in (0, \frac{1}{2}]$ such that
    \begin{itemize}
        \item if $0 < \theta < \frac{1}{2}$, then
                \begin{eqnarray*}
                    \|u - \varphi\|_{H^{1}} + \|u_{t}\|_{L^{2}} = o(t^{-\frac{\theta}{1-2\theta}}), t \rightarrow \infty
                \end{eqnarray*}
        \item if $\theta = \frac{1}{2}$, then
                \begin{eqnarray*}
                    \|u - \varphi\|_{H^{1}} + \|u_{t}\|_{L^{2}} = o(e^{-\zeta t}), t \rightarrow \infty
                \end{eqnarray*}
                with $\zeta > 0$.
    \end{itemize}
\end{main}

\subsection{Example(Integrally positive)}
%\subsubsection*{Example 1}
We consider the Cauchy problem for the nonlinear wave equation with time dependent damping
\begin{eqnarray}
\left
    \{
        \begin{array}{ll}
            \ddot{u} + \frac{h_{0}}{(t+1)^{\alpha}}\dot{u} - \Delta(u) + u^{3} =0 & (t, x) \in \mathbb{R}^{+} \times \mathbb{R}^{N}, \\
                                        u(0,x) = u_{0}(x) & x \in \mathbb{R}^{N},\\
                                  \dot{u}(0,x) = u_{1}(x) & x \in \mathbb{R}^{N},
        \end{array}
\right.  \label{55}
\end{eqnarray}
where $\alpha > 1$, $h_{0} > 0$, $f$ satisfies Assumption 1, 2 and the initial data $(u_{0}(x), u_{1}(x)) \in H^{2}(\Omega) \times H^{1}$ are compactly supported. Obviously, $\frac{h_{0}}{(t+1)^{\beta}}$ is integrally positive, and Eqn.(\ref{55}) satisfies the conditions of Theorem 1 when $p=2$.

%\subsubsection*{Example 2}
%Let $\Omega$ be a bounded open in $\mathbb{R}^{N}, N \geq 1$ with smooth boundary $\partial\Omega$ and let us
%consider the following evolution problem with Dirichlet data:
%\begin{eqnarray}
%\left
%    \{
%        \begin{array}{ll}
%            \ddot{u} + h(t)|\dot{u}|^{\beta-1}\dot{u} - \Delta u +|u|^{\alpha}u= 0 &\textrm{in $\mathbb{R}^{+} \times \Omega$ ;} \\
%                                                        u = 0 &\textrm{on $\mathbb{R}^{+} \times\partial\Omega$;}\\
%                                                   u(0,x) = u_{0}(x) &\textrm{in $\Omega$;}\\
%                                             \dot{u}(0,x) = u_{1}(x) &\textrm{in $\Omega$.}
%        \end{array}
%\right.  \label{56}
%\end{eqnarray}
%Here $ \alpha, \beta \geq 1$.

\subsection{Neumann boundary conditions}

Let $\Omega$ be a bounded, connected set in $\mathbb{R}^{N}, N \geq 1$ with smooth boundary $\partial\Omega$. The exterior normal on $\partial\Omega$ is denoted by $\nu$. We consider the following smilinear wave equation

\begin{eqnarray}
%\left
%    \{
%        \begin{array}{ll}
            \ddot{u} + h(t)\dot{u} - \Delta(u) = f(u) &\textrm{in $\mathbb{R}^{+} \times \Omega$;} \label{57} %\\
%                 \partial_{\nu}u + u + g(\dot{u})= 0 &\textrm{on $\mathbb{R}^{+} \times\partial\Omega$;}\\
%                                        u(0,x) = u_{0}(x) &\textrm{in $\Omega$;}\\
%                                  \dot{u}(0,x) = u_{1}(x) &\textrm{in $\Omega$.}
%        \end{array}
%\right.
\end{eqnarray}
%Here $h(t)$, $f$ are suitably given.
subject to the Neumann boundary condition
\begin{eqnarray}
    \frac{\partial u}{\partial \nu}(t, x) = 0  &\textrm{on $\mathbb{R}^{+} \times\partial\Omega$;}  \label{58}
\end{eqnarray}
and the initial condition
\begin{eqnarray}
    u(0,x) = u_{0}(x), \dot{u}(0,x) = u_{1}(x) &\textrm{in $\Omega$.}  \label{59}
\end{eqnarray}
\begin{main}
    Assume $h$ is integrally positive or in positive-negative case and $f$ satisfies Assumption 1 and 2. Let $(u, u_{t})^{T}$ be a solution of problem (\ref{57})--(\ref{59}) with $(u_{0}, u_{1})^{T} \in \mathcal{D}$, and its trajectory is precompact in $H^{2}(\Omega) \times H^{1}(\Omega)$. Then there exists a equilibrium $(\varphi, 0)$, with $\varphi$ in the set
    \begin{eqnarray*}
        \Sigma = \{ \varphi \in H^{2}(\Omega)\cap H^{1}(\Omega):  - \Delta \varphi = f(\varphi) \},
    \end{eqnarray*}
    and $\theta = \theta(\varphi) \in (0, \frac{1}{2}]$ such that
    \begin{itemize}
        \item if $0 < \theta < \frac{1}{2}$, then
                \begin{eqnarray*}
                    \|u - \varphi\|_{H^{1}} + \|u_{t}\|_{L^{2}} = o(t^{-\frac{\theta}{1-2\theta}}), t \rightarrow \infty
                \end{eqnarray*}
        \item if $\theta = \frac{1}{2}$, then
                \begin{eqnarray*}
                    \|u - \varphi\|_{H^{1}} + \|u_{t}\|_{L^{2}} = o(e^{-\zeta t}), t \rightarrow \infty
                \end{eqnarray*}
                with $\zeta > 0$.
    \end{itemize}
\end{main}

\subsection{Dynamical boundary conditions}

Let $\Omega$ be a bounded, connected set in $\mathbb{R}^{N}, N \geq 1$ with smooth boundary $\partial\Omega$. The exterior normal on $\partial\Omega$ is denoted by $\nu$. We consider the following smilinear wave equation
\begin{eqnarray}
%\left
%    \{
%        \begin{array}{ll}
            \ddot{u} + h(t)\dot{u} - \Delta(u) = f(u) &\textrm{in $\mathbb{R}^{+} \times \Omega$;} \label{60} %\\
%                 \partial_{\nu}u + u + g(\dot{u})= 0 &\textrm{on $\mathbb{R}^{+} \times\partial\Omega$;}\\
%                                        u(0,x) = u_{0}(x) &\textrm{in $\Omega$;}\\
%                                  \dot{u}(0,x) = u_{1}(x) &\textrm{in $\Omega$.}
%        \end{array}
%\right.
\end{eqnarray}
%Here $h(t)$, $f$ are suitably given.
subject to the dynamical boundary condition
\begin{eqnarray}
    \partial_{\nu}u + u + \dot{u}= 0 &\textrm{on $\mathbb{R}^{+} \times\partial\Omega$;}  \label{61}
\end{eqnarray}
and the initial condition
\begin{eqnarray}
    u(0,x) = u_{0}(x), \dot{u}(0,x) = u_{1}(x) &\textrm{in $\Omega$.}  \label{62}
\end{eqnarray}
Here $f$ satisfies the assumption 2 in section 2, and $h$ is integrally positive or in the positive-negative case.

We equip $H^{1}(\Omega)$ with norm
\[
    \| u \|_{H^{1}} = ( \int_{\Omega} |\nabla u|^{2} \mathrm{d}x + \int_{\partial\Omega} u^{2}\mathrm{d}S )^{\frac{1}{2}}
\]

Denote
\begin{eqnarray*}
    \mathcal{D} = \{ (u, v)^{T} \in H^{2}(\Omega) \times H^{1}(\Omega):  \partial_{\nu}u + u + v |_{\partial\Omega}= 0 \}
\end{eqnarray*}
which is clearly a closed subspace of $ H^{2}(\Omega) \times H^{1}(\Omega)$.
For each solution $u$ of problem(1), we define its energy by
\begin{eqnarray*}
    E_{u}(t) =  \int_{\Omega}\frac{1}{2}( |u_{t}|^{2} + |\nabla(u)|^{2} )\mathrm{d}x + \frac{1}{2} \int_{\partial\Omega} u^{2}\mathrm{d}S- F(u)\mathrm{d}x
\end{eqnarray*}
where $F(u)\triangleq \int_{0}^{u}f(s)\mathrm{d}s$. In addition, we denote
\begin{eqnarray*}
    e_{u}(t)=\int_{\Omega}\frac{1}{2}|\nabla(u)|^{2}\mathrm{d}x + \frac{1}{2} \int_{\partial\Omega} u^{2}\mathrm{d}S  - F(u)\mathrm{d}x.
\end{eqnarray*}
If there is no need to specify $u$, we simplify $E_{u}(t), e_{u}(t)$ by $E(t), e(t)$ respectively.

In this case we have the following
\begin{main}
     For any initial data $(u_{0}, u_{1})^{T} \in \mathcal{D}$,  problem (\ref{60})--(\ref{62}) admits a unique global solution
     \[
        (u, u_{t})^{T} \in C([0, T]; \mathcal{D})\cap C^{1}([0, T]; \mathcal{H}).
     \]
     Moreover, $(u, u_{t})^{T}$ converges to an equilibrium $(\psi, 0)^{T}$ in the topology of $\mathcal{H}$ as time goes to infinity, i.e.,
    \begin{eqnarray*}
        \lim_{t \rightarrow \infty}{\|u_{t}\| + \|u - \psi\|_{H^{1}}} = 0.
    \end{eqnarray*}
\end{main}
Here $\psi$ is an an equilibrium to problem (\ref{60})--(\ref{62}), i.e. $\psi$ is a classical solution to the following nonlinear elliptic boundary value problem:
    \begin{eqnarray}
    \left
        \{
            \begin{array}{ll}
                -\triangle \psi = f(\psi) &\textrm{in $\Omega$;} \\
                \partial_{\nu}\psi + \psi =0   &\textrm{on $\partial\Omega$;} \label{63}
            \end{array}
    \right.
    \end{eqnarray}

\begin{remark}
    As we know, the time-dependent damping defined in definition 2 may vanish in some interval. However due to the boundary dissipation, we can also get the convergence of the solutions. This case will be studied deeply in \cite{20}.
\end{remark}

\begin{lemma}
    Let $\psi$ is an an equilibrium to problem (\ref{60}), then there exist $\theta \in (0, \frac{1}{2}]$ and $\delta > 0$ such that
    $\forall u\in H^{2}$, $ \|u - \psi\|_{H^{2}} < \delta$,
        \begin{eqnarray*}
            \|-\triangle u + f(u)\|_{L^{2}(\Omega)} + \|\partial_{\nu} u + u\|_{L^{2}(\partial\Omega)} \geq \mid e_{u} - e_{\psi}\mid^{1-\theta}.
        \end{eqnarray*}
\end{lemma}
The proof of Theorem 4 is the same as Theorem 1 and 2, but we have to use the generalized {\L}ojasiewicz-Simon inequality(Lemma 6) established in [19].
\\

\end{document}